\documentclass{amsart}

\usepackage{amscd,amsxtra,amsthm}
\usepackage[all]{xy}
\usepackage{etex}
\usepackage{pictex}
\usepackage{graphicx}
\usepackage{mathtools}
\usepackage{float}

\usepackage{multirow}
\DeclarePairedDelimiter{\ceil}{\lceil}{\rceil}

\usepackage{comment}

\newtheorem{theorem}{Theorem}

\theoremstyle{definition}
\newtheorem{definition}[theorem]{Definition}

\theoremstyle{corollary}
\newtheorem{corollary}[theorem]{Corollary}

\theoremstyle{conjecture}

\theoremstyle{remark}

\theoremstyle{proposition}
\newtheorem{proposition}[theorem]{Proposition}

\begin{document}
\parskip10pt
\parindent15pt
\baselineskip15pt    

\title{Minimally Connected Hypergraphs}

\author[M. Budden]{Mark Budden}
\address{Department of Mathematics and Computer Science \\
Western Carolina University \\
Cullowhee, NC 28723 USA}
\email{mrbudden@email.wcu.edu}

\author[J. Hiller]{Josh Hiller}
\address{Department of Mathematics and Computer Science \\ 
Adelphi University \\
Garden City, NY 11530-0701}
\email{johiller@adelphi.edu}

\author[A. Penland]{Andrew Penland}
\address{Department of Mathematics and Computer Science \\
Western Carolina University \\
Cullowhee, NC 28723 USA}
\email{adpenland@email.wcu.edu}


\begin{abstract}
Graphs and hypergraphs are foundational structures in discrete mathematics.  They have many practical applications, including the rapidly developing field of bioinformatics, and more generally, biomathematics. They are also a source of interesting algorithmic problems. In this paper, we define a \textit{construction process} for minimally connected $r$-uniform hypergraphs, which captures the intuitive notion of building a hypergraph piece-by-piece, and a numerical invariant called the \textit{tightness}, which is independent of the construction process used. Using these tools, we prove some fundamental properties of minimally connected hypergraphs. We also give bounds on their chromatic numbers and provide some results involving edge colorings.  We show that every connected $r$-uniform hypergraph contains a minimally connected spanning subhypergraph and provide a polynomial-time algorithm for identifying such a subhypergraph. \end{abstract}

\maketitle


\section{Introduction}\label{intro}

Graphs and hypergraphs provide many beautiful results in discrete mathematics. They are also extremely useful in applications. Over the last six decades, graphs and their generalizations have been used for modeling many biological phenomena, ranging in scale from protein-protein interactions, individualized cancer treatments, carcinogenesis, and even complex interspecial relationships \cite{GWVP, HVBK, KHT, LBAA, WJ}. In the last few years in particular, hypergraphs have found an increasingly prominent position in the biomathematical literature, as they allow scientists and practitioners to model complex interactions between arbitrarily many actors \cite{WJ}.  

The flexibility that makes hypergraphs such a versatile tool complicates their analysis.  Because of this, many different algorithms and metrics have been developed to assist with their application  \cite{LBBGE, ZCYK}. Due to the sheer number of researchers from varying disciplines developing these techniques, there is substantial inconsistency in the literature regarding names and notations. However, many of these approaches share a common theme: they aim to quantify or model connectivity in some way \cite{LBBGE}. 

Trees, and in particular, spanning trees, offer a very useful tool for studying connectivity in graphs. Indeed, trees play an all-important role in combinatorics: they are simple enough to provide intuition via examples, yet  sufficiently complex to provide richness and depth. Trees also illustrate some differences between graphs and hypergraphs. For instance, it is well-known that every connected graph has a spanning tree. It is also well-established that this statement is not true for hypergraphs.


The distinction between graphs and other hypergraphs with regard to spanning trees has important considerations in theoretical computer science. If one needs a spanning tree in a graph, standard algorithms such as that of Prim~\cite{P} or Kruskal~\cite{K}  will do the job in low-degree polynomial time. For $3$-uniform hypergraphs, an algorithm due to Lov\'asz~\cite{LL} will also determine the existence of a spanning tree in polynomial time. A subsequent, more efficient polynomial time algorithm for the same problem is due to Gabow and Stallman~\cite{GS}. However, Andersen and Fleischner~\cite{AF} showed that the general problem of determining whether or not a hypergraph has a spanning tree is NP-complete, even for relatively restricted classes, such as linear hypergraphs in which each vertex is contained in at most 3 hyperedges, or 4-uniform hypergraphs which have some vertex in common to all hyperedges. Andersen and Fleischner \cite{AF} quote this last fact as an unpublished result of Carsten Thomassen. 

Another important distinction between graphs and hypergraphs arises in the equivalence of certain definitions of a spanning tree. In graph theory, every spanning minimally connected subgraph is a tree, but the analagous statement does not hold for hypergraphs. This distinction between spanning trees and minimally connected subhypergraphs appears when generalizing certain results from graphs to hypergraphs.  For example, in 2014, Chartrand, Johns, McKeon, and Zhang \cite{CJMZ} proved that a connected graph $G$ has its rainbow connection number equal to its size if and only if $G$ is a tree.  In order to prove an analogue of this result in the setting of hypergraphs, Carpentier, Liu, Silva, and Sousa \cite{CLSS} were forced to consider the more general class of minimally connected hypergraphs.  In this paper, we examine other aspects of hypergraph theory where minimally connected hypergraphs are necessary to prove results that typically concern trees in graphs.  


In the next section, we give the formal definition of an $r$-uniform tree and provide a very simple demonstration that for every $r>2$, there exist connected hypergraphs which do not admit spanning $r$-uniform trees. In Section 3, we consider  structural properties and existence theorems for minimally connected hypergraphs, introducing a numerical invariant called \textit{tightness} associated to a hypergraph. In Section 4, we examine chromatic numbers of minimally connected hypergraphs. Section 5 of this paper deals with other connectivity issues related to minimally connected hypergraphs. We conclude with some open questions and  directions for future research. 

\section{Definitions and Background}\label{defback}

In this section, we provide the definitions, elementary examples, and concepts that are necessary to derive the results of this paper. We begin by formally defining hypergraphs and $r$-uniform hypergraphs.  As is customary, we will denote the cardinality of a set $S$ by $|S|.$ If $S$ is a set, we write $2^S$ for the power set of $S$. 

A hypergraph $H$ consists of two sets: a non-empty set $V(H)$ called the \textit{vertex set}, and a set $E(H) \subseteq 2^V - \emptyset$, called the set of \textit{hyperedges}. When the hypergraph being considered is clear from the context, we may write $V$ and $E$ in place of $V(H)$ and $E(H)$, respectively.  The \textit{size} of a hypergraph is  $|E|$ while the order of a hypergraph is  $|V|$. An $r$-uniform hypergraph is a hypergraph where for all $e\in E$, $|e|=r$. When $r=2$, our definition coincides with that of a graph. For every hypergraph $H$, there is a corresponding hypergraph $\overline{H}$ such that $V(H)=V(\overline{H})$ and $E(\overline{H})=2^V-(E(H)\cup \{\emptyset \})$. If $H$ is assumed to be an $r$-uniform hypergraph, then we assume $\overline{H}$ is also an $r$-uniform hypergraph, and so we only look at the edge complement of $H$ within the more restricted set of vertex sets of cardinality $r$.   

Throughout the remainder of this paper, we will focus on $r$-uniform hypergraphs with $r>2$. We call a hypergraph $H$ {\it finite} if $V(H)$ is finite. For the remainder of this paper, we will only consider finite hypergraphs. For clarity, when we refer to graphs, we will call them $2$-graphs.  

A {\it Berge path} consists of a sequence of $k$ distinct vertices $v_1$, $v_2$, \dots , $v_k$ and $k-1$ distinct hyperedges $e_1$, $e_2$, \dots , $e_{k-1}$ such that $v_i, v_{i+1}\in e_i$ for all $i\in \{1, 2, \dots , k-1\}$.  A {\it Berge cycle} is formed if there is a hyperedge $e_k$ that includes both $v_1$ and $v_k$.  A Berge path is a {\it loose path} if for $i\not=j$, $$|e_i\cap e_{j}|=\left\{ \begin{array}{ll} 0 & \mbox{if $j\ne i+1$} \\ 1 & \mbox{if $j=i+1$}.\end{array}\right.$$ Observe that all vertices in a loose path are necessarily distinct.


We say that an $r$-uniform hypergraph $H$ is {\it minimally connected} if the removal of any hyperedge (while retaining all vertices) disconnects $H$.   For example, consider the hypergraphs in Figure \ref{minimal}. Every hyperedge in the first hypergraph contains some vertex of degree one, but this is not the case in the second hypergraph.
\begin{figure}[H]
\centerline{
{\includegraphics[width=0.53\textwidth]{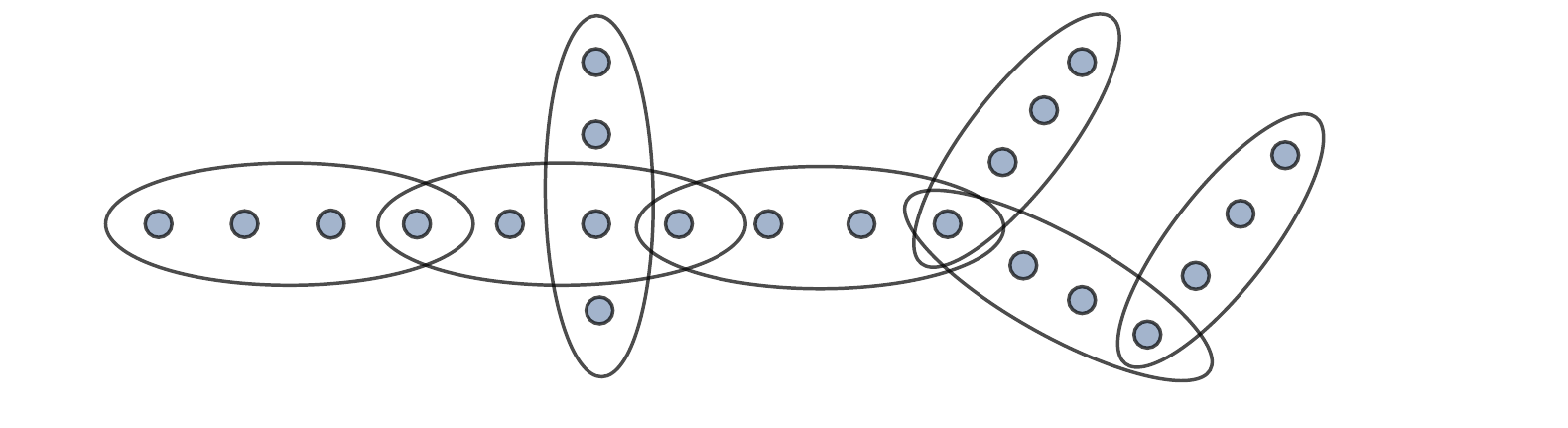}}{\includegraphics[width=0.53\textwidth]{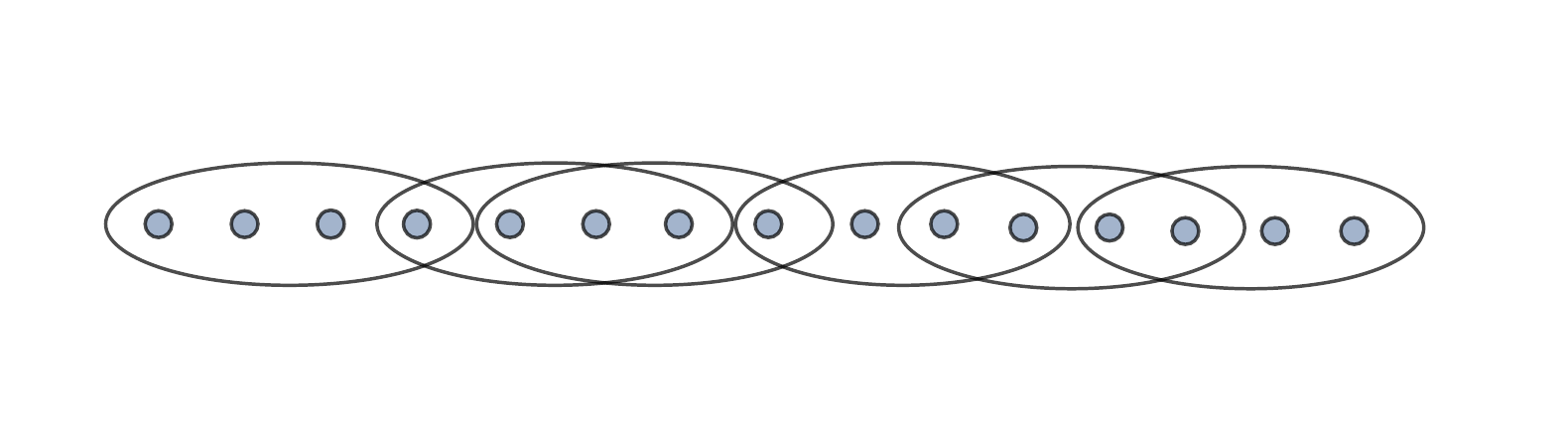}}}\caption{Two minimally connected $4$-uniform hypergraphs.  } \label{minimal}
\end{figure} 
\noindent Observe that if every hyperedge of a connected hypergraph contains at least one vertex of degree $1$, then it is minimally connected.  
The second hypergraph in Figure \ref{minimal} demonstrates that the converse to this statement is false. Figure \ref{minimal2} shows two other minimally connected hypergraphs, both of which are Berge cycles.  
\begin{figure}[h!]
\centerline{
{\includegraphics[width=0.4\textwidth]{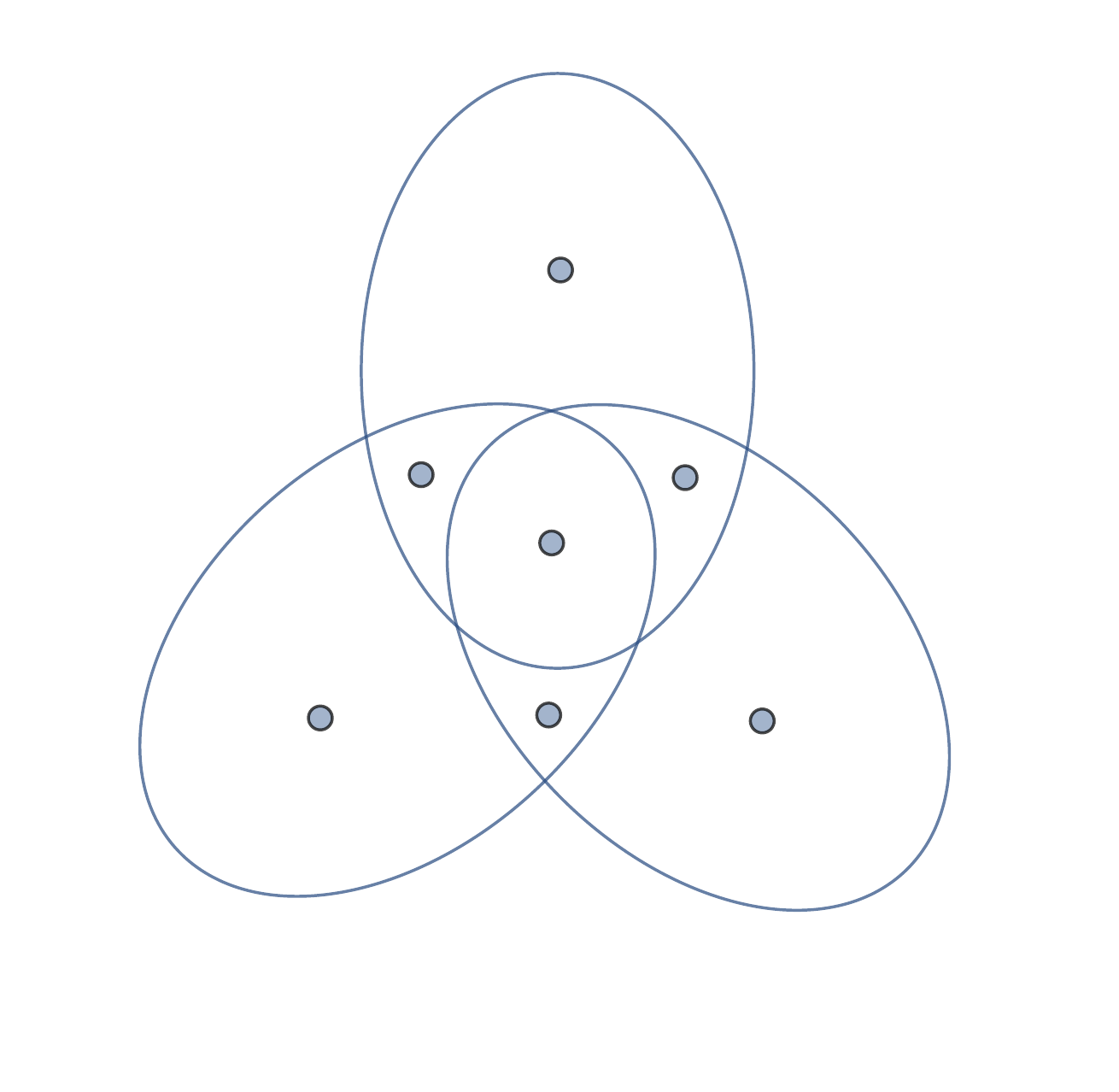}}{\includegraphics[width=0.4\textwidth]{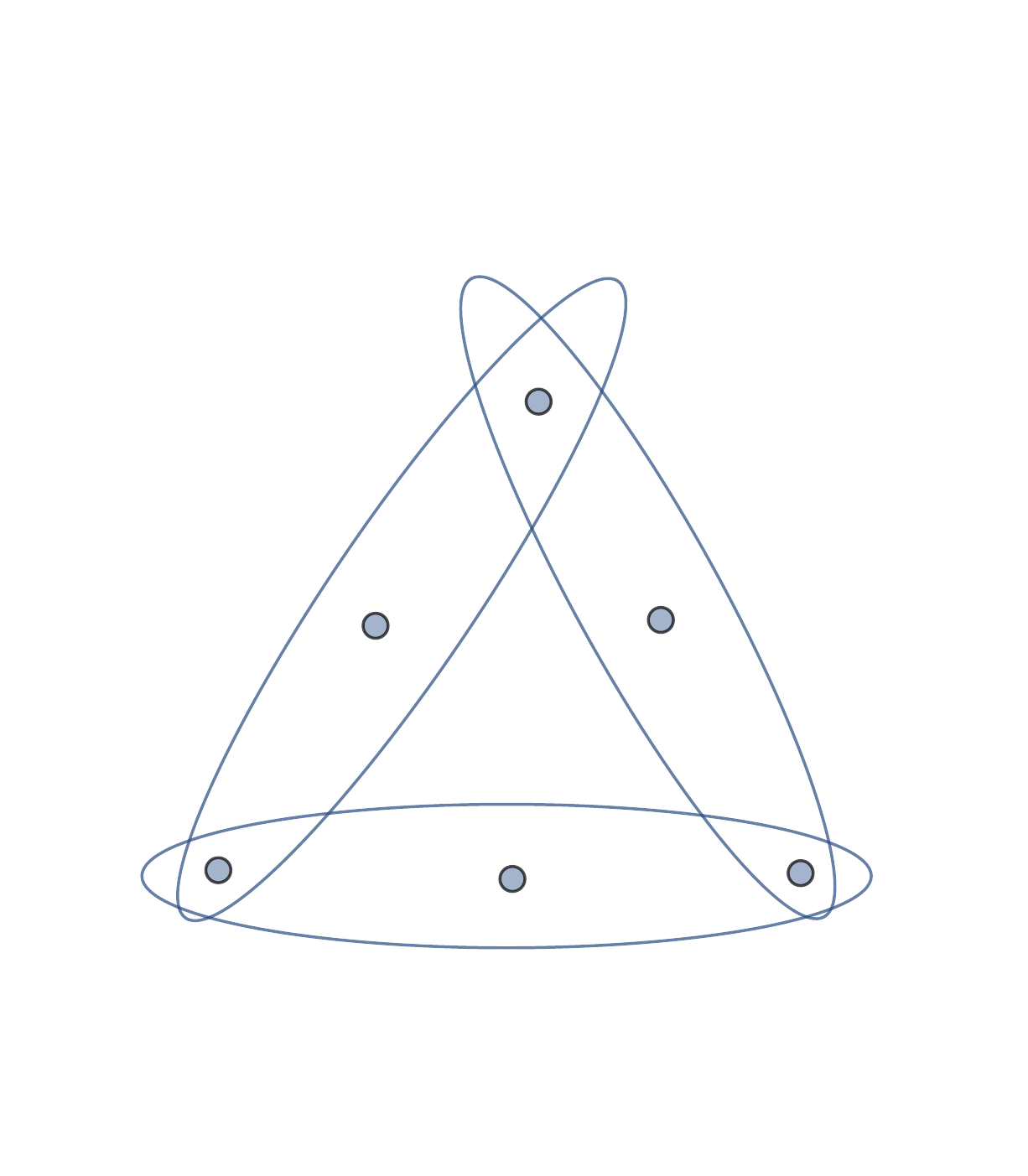}}}\caption{Two minimally connected hypergraphs in which any hyperedge-by-hyperedge construction requires the third hyperedge to intersect at least two previous hyperedges.}\label{minimal2}
\end{figure} 

Of course, having defined minimally connected hypergraphs, one must inquire about the appropriate definition of a hypergraph tree. We now give several possible definitions.

\begin{definition}\label{tree} The following definitions of $r$-uniform trees are equivalent:
\begin{enumerate}
\item $T$ is an $r$-uniform hypergraph that can be formed hyperedge-by-hyperedge with each new hyperedge intersecting the previous hypergraph at exactly one vertex.  That is, each new hyperedge requires the creation of exactly $r-1$ new vertices.
\item $T$ is a connected $r$-uniform hypergraph that does not contain any (Berge) cycles.
\item $T$ is a connected $r$-uniform hypergraph in which the removal of any hyperedge (keeping all vertices) results in a hypergraph with exactly $r$ connected components.
\item $T$ is an $r$-uniform hypergraph in which there exists a unique loose path between any pair of distinct vertices.
\item $T$ is a connected $r$-uniform hypergraph in which the size $|E|$ and order $|V|$ satisfy $|V|=(r-1)|E|+1$.
\end{enumerate}
\end{definition}

The equivalence of $(1)-(4)$ can be found in Theorem 2.1 of \cite{BP}. We will wait until Corollary \ref{treedef} in Section 3 to complete the proof that $(5)$ is also a suitable definition.
An important observation is that  from $(3)$, it immediately follows that every $r$-uniform tree is minimally connected (in that the removal of any hyperedge disconnects the hypergraph).  Of course, the hypergraphs given in Figures 1 and 2 show that not all minimally connected hypergraphs are trees. 

The examples given so far demonstrate that a hyperedge-by-hyperedge construction must allow for the intersection of a new hyperedge with more than one vertex (and even more than one hyperedge) in the previous hypergraph.  The examples in Figure \ref{minimal2} also demonstrates why it is necessary to allow for (Berge) cycles.  Of course, if cycles are allowed, then there can be multiple paths between a pair of distinct vertices. In fact, the following proposition shows that it is quite easy to construct such examples for every $r>2$. Let us first define an \textit{ $r$-uniform hypergraph spanning tree} for a an $r$-uniform hypergraph $H$ to be an $r$-uniform tree $T$ with $V(T)=V(H)$ and $E(T)\subseteq E(H).$

\begin{proposition}
For every $r>2$ there exists a connected $r$-uniform hypergraph which does not admit an $r$-uniform spanning tree.
\end{proposition}

\begin{proof}
Fix $r>2$ and consider the $r$-uniform hypergraph on $r+1$ vertices with two edges. $e_1$ consists of vertices $(v_1,v_2,...v_r)$ and $e_2$ consists of vertices $(v_2,v_3,...,v_{r+1}).$ Then the hypergraph obtained is connected and not a tree but the removal of either edge disconnects the hypergraph. 
\end{proof}


%

\section{Structural Properties and Existence Theorems}\label{prop}

In this section, we analyze the underlying structure of minimally connected hypergraphs.  We start by defining a construction process, which leads to the notion of the tightness sum of such a hypergraph.  Our attention then turns to possible sizes and an algorithm for finding a spanning minimally connected subhypergraph for any connected $r$-uniform hypergraph.

\subsection{Construction Processes for Hypergraphs}

The definition we offered for a minimally connected hypergraph can be a little fastidious  to work with when proving basic properties of  hypergraphs. Hence, we wish to offer an equivalent constructive definition of a minimally connected $r$-uniform hypergraph that will serve our purposes nicely. In order to make this notion precise, we formally define a \textit{constructive process}. 

\begin{definition}
Let $H$ be an $r$-uniform hypergraph of size $n$. A \textit{constructive process} $\mathcal{P}$ for $H$ is a finite sequence of hypergraphs $H_i$, $1 \leq i \leq n$ satisfying the following properties: \\
\begin{itemize}

\item  $H_n = H$,
\item $H_j$ is a subhypergraph of $H_{k}$ for $j \leq k$, 
\item each $H_i$ has exactly $i$ hyperedges, and 
\item for each $j$, $E(H_{j+1}) = E(H_j) \bigcup \{ e_j \} $ for some $e_j \in E(H)$.
\end{itemize}

We say that a constructive process $\mathcal{P}$ for a hypergraph $H$ is {\it connected} if and only if each $H_i$ is connected for all $i$.  It is clear that a hypergraph is connected if and only if it has a connected construction process.  We say that a constructive process $\mathcal{P}$ for a hypergraph $H$ is {\it minimally connected} if and only if each $H_i$ is minimally connected. 

\end{definition}


\begin{theorem}  
An $r$-uniform hypergraph $H$ is minimally connected if and only if 
$H$ is an $r$-uniform hypergraph that has a minimally connected constructive process. 
\label{equiv}
\end{theorem}

\begin{proof}
First, suppose that $H$ a is minimally connected $r$-uniform hypergraph.  In particular, $H$ is connected, so it can be constructed hyperedge-by-hyperedge with each resulting hypergraph being connected along the way.  Let $H_i$ be the resulting hypergraph after the first $i$ hyperedges $e_1, e_2, \dots , e_i$ have been added.  Note that if $H_i$ is not minimally connected, then there exists some hyperedge whose removal does not disconnect $H_i$, and hence, would not disconnect $H$.  Conversely, if a hypergraph $H$ can be formed hyperedge-by-hyperedge with the resulting hypergraph being minimally connected at each stage, then the final stage results in $H$, which is necessarily minimally connected.
\end{proof}

We have shown that there is a constructive process $\mathcal{P}$ that one can use to obtain a minimally connected hypergraph.  Suppose that $H$ is a minimally connected $r$-uniform hypergraph of size $k$.  Denote the hyperedges in this constructive process by $e_1$, $e_2,$ \dots , $e_k$ and let $H_i$ be the resulting minimally connected hypergraph after hyperedge $e_i$ has been added.  Note that with the addition of each hyperedge, at least one new vertex must be introduced.  Otherwise, when adding in a hyperedge $e_i$ that only uses existing vertices, the hypergraph $H_{i-1}$ would have to have been disconnected, contradicting our assumption about the constructive process.  It is also worth observing that the addition of any $e_i$ cannot prevent the removal of a previous hyperedge $e_j$ ($j<1$) from disconnecting the hypergraph (although it may change the number of resulting components).

\subsection{The Notion of Tightness And Some Applications}
A constructive process $\mathcal{P}$ produces a sequence $t_1$, $t_2$, \dots , $t_{k-1}$ of ``tightnesses'' given by $$t_i=| V(H_{i})\cap e_{i+1}|,$$ where $1\le t_i \le r-1$.  Although the constructive process that we have described is not unique for a given minimally connected hypergraph $H$, we will show in the following theorem that the sum of the tightnesses is independent of the construction chosen.

\begin{theorem}
All constructive processes for a given $r$-uniform hypergraph $H$ of size $k$ produce the same tightness sum $$t_{H}:=\mathop{\sum}\limits_{i=1}^{k-1} t_i.$$
\end{theorem}

\begin{proof}
Let $H$ be a connected $r$-uniform hypergraph of size $k$ and suppose that $\mathcal{P}$ and $\mathcal{P}'$ are constructive processes for $H$ with tightness sequences $$t_1, t_2, \dots , t_{k-1} \quad \mbox{and} \quad t_1', t_2', \dots t_{k-1}',$$ respectively.  Then for $1\le i\le k-1$, the addition of $e_{i+1}$ in $\mathcal{P}$ requires the addition of $r-t_i$ new vertices.  An analogous statement can be made for $\mathcal{P}'$.  If $H$ has order $n$, then $$n=r+\mathop{\sum}\limits_{i=1}^{k-1} (r-t_i) =r+\mathop{\sum}\limits_{i=1}^{k-1} (r- t_i'),$$ from which it follows that $$\mathop{\sum}\limits_{i=1}^{k-1} t_i=\mathop{\sum}\limits_{i=1}^{k-1} t_i',$$ completing the proof of the theorem.
\end{proof}

Thus, $t_H$ is uniquely determined by $H$ and the order of $H$ is given by  $$|V(H)|=r+\mathop{\sum}\limits_{i=1}^{k-1} (r-t_i) =rk-t_H.$$ The following corollary will prove the equivalence between $(3)$ and $(5)$ in Definition \ref{tree} .  

\begin{corollary}
A connected $r$-uniform hypergraph $H$ is an $r$-uniform tree if and only if $t_{H}=k-1$.  Equivalently, a connected $r$-uniform hypergraph $H$ is an $r$-uniform tree if and only if it has order $|V(H)|=(r-1)|E(H)|+1$.  \label{treedef}\end{corollary}


\begin{proof}
Clearly every hypergraph tree is minimally connected and has size $k$ if and only if it has order $r+(k-1)(r-1)$.  Thus, we need only show that if $H$ is a connected $r$-uniform hypergraph of size $k$ and order $r+(k-1)(r-1),$ then $H$ is a tree.  We prove the contrapositive to this statement.  Suppose that $H$ is a connected $r$-uniform hypergraph that is not a tree.  Then $H$ can be constructed hyperedge-by-hyperedge, with the resulting hypergraph being connected at each stage of the construction.  Since $H$ is assumed to not be a tree, some tightness $t_i\ge 2$.  So, the sum of the tightnesses of $H$ satisfies $t_H\ge k$, giving a maximal order of $r+(k-1)(r-1)-1$.
\end{proof}

Denote by $S_n^{(r)}$ the $r$-uniform star of order $n$ consisting of $r-1$ vertices in the intersection of all hyperedges (called the {\it center}), along with each hyperedge containing a single vertex of degree $1$.  This definition agrees with the more general definititon of a star given in \cite{BHR2}.  Observe that such a star is minimally connected since the removal of any hyperedge leaves the vertex of degree $1$ isolated. See Figure \ref{star} for an example of a $6$-uniform star.
\begin{figure}[h!]
\centerline{
{\includegraphics[width=0.6\textwidth]{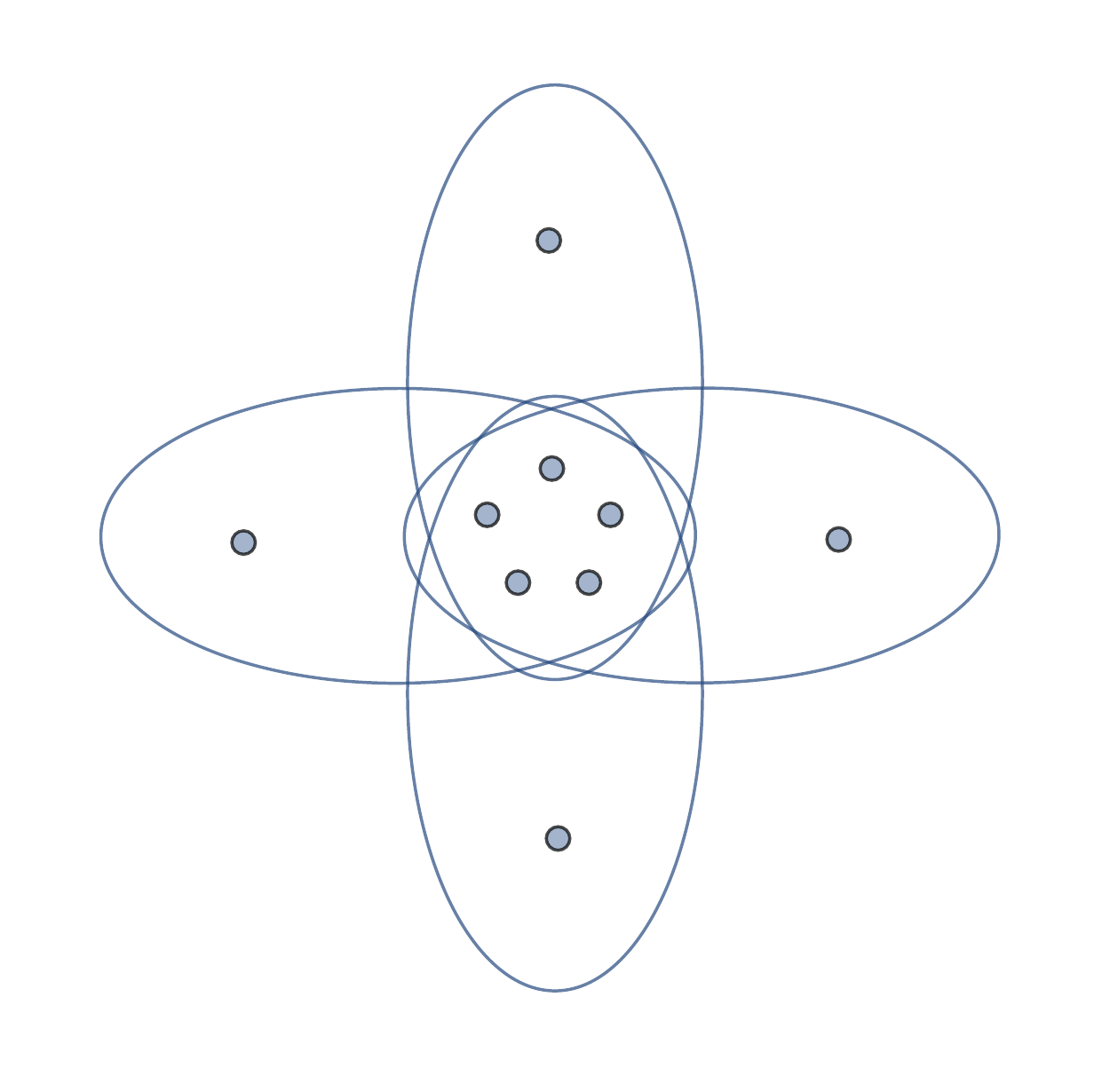}}}
\caption{A $6$-uniform star of order $9$.}\label{star}
\end{figure} 

\begin{theorem}
For $r\geq 3$, a minimally connected $r$-uniform hypergraph $H$ of order $n$ is isomorphic to $S_n^{(r)}$ if and only if $$t_H=(k-1)(r-1).$$
\end{theorem}

\begin{proof}
If a minimally connected hypergraph $H$ is a star, then every constructive process for $H$ requires each tightness $t_i=r-1$, giving $t_H=(k-1)(r-1)$.  To prove the converse, we will use induction on $k$.  When $k=1$, $t_H=0$ and $H$ consists of a single hyperedge, which is trivially a star.  Now suppose that $H$ is a minimally connected hypergraph of size $k>1$ in which $t_H=(k-1)(r-1)$ and that all minimally connected hypergraphs of size $k-1$ (with tightness sum equal to $(k-2)(r-1)$) are necessarily stars.  As a minimally connected hypergraph, $H$ has a minimally connected constructive process $\mathcal{P}$.  Let $e_k$ be the last hyperedge added in such a process and consider the hypergraph $H'$ formed by removing $e_k$ (and all resulting singletons) from $H$.  Since $t_H=(k-1)(r-1)$, all tightnesses $t_i$ in $\mathcal{P}$ are equal to $r-1$.  Hence, the removal of $e_k$ leaves only one singleton and we find that $|V(H)|=|V(H')|+1$.  So, \begin{align}t_{H'}&=r(k-1)-|V(H')|\notag \\
&=r(k-1)-(|V(H)|-1)\notag \\
&=r(k-1)-(rk-t_H-1)\notag \\
&=r(k-1)-(rk-(k-1)(r-1)-1) \notag \\
&=(k-2)(r-1).\notag \end{align}  By the inductive hypothesis, $H'$ is a star.  When adding back in hyperedge $e_k$ to form $H$, including any vertex $x$ of degree one will prevent the hyperedge $e_i$ that includes $x$ from disconnecting $H$.  Thus, the only vertices that can be included in the intersection of $e_k$ with $H'$ are those of degree greater than one in $H'$.  Therefore, $H$ is a star.
\end{proof}

\subsection{The Possible Sizes of Minimally Connected Subhypergraphs}

We now focus our attention on finding bounds for the size of a minimally connected $r$-uniform hypergraph. Let $H$ be a connected hypergraph of order at least $r+1$.  From Theorem \ref{equiv}, each minimally connected spanning hypergraph $M$ can be constructed hyperedge-by-hyperedge, with each resulting hypergraph being minimally connected, and this process results in a (finite) sequence of tightnesses $t_2, t_3,\dots, t_m$, where $1\le t_i  \le r-1$ and $m$ is the size of $M$.  The maximum number of hyperedges $M$ can contain occurs when $M$ is $(r-1)$-tight (i.e., $t= t_2, =t_3=\cdots t_m=r-1$).  In this case, we find that $M$ contains $m=n-r+1$ hyperedges.  For example, consider the star in Figure \ref{star}.

The minimum number of hyperedges that $M$ can contain occurs when $M$ is a tree, or is close to being a tree (with say, only one tightness not equal to $1$).  If $M$ is a tree that has order $n$, then $n-1\equiv 0\pmod{r-1}$.  If $m$ is the size of $M$, then $$m=\frac{n-1}{r-1}.$$ If $M$ is not a tree, but is close to being a tree, then suppose that $$n-1= (m-1)(r-1)+k, \quad \mbox{where} \quad 1\le k\le r-2.$$  It follows that $$m=\frac{n-1}{r-1}+\frac{r-1-k}{r-1}=\ceil[\Big]{\frac{n-1}{r-1}}.$$  Putting together these upper and lower bounds, we have shown the following.

\begin{theorem}
Let $H$ be a connected $r$-uniform hypergraph of order $n$ with minimally connected spanning hypergraph $M$ of size $m$.  Then $$\ceil[\Big]{\frac{n-1}{r-1}}\le m \le n-r+1.$$ \label{size}
\end{theorem}

Theorem \ref{size} hints at a question: for a given order $n$ is there a minimally connected $r$-uniform hypergraph for every permissible value of $m$?  The following result shows that this is the case.

\begin{theorem}
Fix an order $n\ge r$ and let $$\ceil[\Big]{\frac{n-1}{r-1}}\le m \le n-r+1.$$  \label{size2} Then there is a minimally connected $r$-uniform hypergraph of size $m$ and order $n$. 
\end{theorem}

\begin{proof}
Let $z=n-r+1-m$. We will define a constructive process which culminates in an $r$-uniform hypergraph of size $m$ and order $n$. Let us choose $r$ vertices to form $e_1$ and label these $v_1,v_2,...,v_r$. We then proceed in two cases based on the value of $z$:\\
\underline{Case 1:} \ If $z\geq r-2,$ then define $a,b\in \mathbb{Z}$ such that $z=a(r-2)+b.$ For $1\leq i \leq a$ let $t_i=1,$ $t_{a+1}= r-b-1,$ and $t_j=r-1$ for $a+1<j<m-1$. For $e_2$ let $|e_1\cap e_2|=\{v_1\}.$ For $e_k$ ($2<k<m)$ pick any $t_{i-1}$ vertices from $v_2,v_3,...,v_r.$ Then every edge, except possibly $e_1$ contains a vertex of degree $1$. If $e_1$ is removed then $e_2$ is disconnected from the rest of the hypergraph. Thus the resulting hypergraph is of size $k$, order $n$ and is minimally connected.\\
\underline{Case 2:} \ If $z<r-2,$ then let $t_1=r-z-1.$ Let $|e_1\cap e_2|= \{v_1,v_2,...,v_{t_1}\}$ create $m-2$ edges $e_3,e_4,...,e_m$ and for every $2<i\leq m$, let $|e_1\cap e_i|=\{v_1,v_2,...,v_{r-1}\}.$ Thus every edge has at least one vertex of degree $1$, and so the resulting hypergraph is of size $k$, order $n$ and is minimally connected.  \\
\noindent This completes the proof of the theorem. 
\end{proof}

Theorem \ref{size2} leads to an interesting observation. The complete hypergraph $K_n^{(r)}$ contains minimally connected subhypergraphs for every permissible value of $m$. This fact stands in stark contrast to the context of $2$-graphs, where it is well known that the size of a spanning tree is completely determined by the order of the parent graph.

\subsection{Existence and An Algorithm for Minimally Connected Subhypergraphs}

The original motivation for this paper was to generalize a ubiquitous result in the study of trees: every connected graph contains a spanning tree. We now provide an analogous result for $r$-uniform hypergraphs and describe a polynomial-time algorithm for finding such an subhypergaph.

\begin{theorem}
Every connected $r$-uniform hypergraph contains a spanning minimally connected hypergraph. \label{span}
\end{theorem}

\begin{proof}
Let $H$ be a connected $r$-uniform hypergraph.  If $H$ is not minimally connected, then there exists some hyperedge $e_1$ whose removal does not disconnect $H$.  Let $H_1$ be the hypergraph formed by removing $e_1$ from $H$.  If $H_1$ is not minimally connected, then repeat this process.  As $H$ has a finite number of hyperedges, the process must eventually terminate with a minimally connected hypergraph $H_i$ that spans the vertices in $H$.
\end{proof}

In this subsection we discuss an algorithm for finding a minimally connected spanning subhypergraph of a graph. In order to make the discussion clear, we will give necessary definitions. 

An \textit{algorithm} is an unambiguous, step-by-step process that takes input and terminates with some output. If $f$ and $g$ are functions from $\mathbb{N}$ to itself, we say that $g(n)$ is \textit{ $O(f(n))$} if there exist positive constants $c$ and $n_0$ such that $0 \leq g(n) \leq cf(n)$ for all $n \geq n_0$. For an algorithm $A$, we let $T_A(n)$ denote the maximum number of steps that it takes $A$ to terminate on an algorithm of size $n$; note that $T_A$ is a function from $\mathbb{N}$ to $\mathbb{N}$. Given an algorithm $A$ and a function $f(n)$, we say that an algorithm $A$ is $O(f(n))$ if the function $T_A$ is $O(f(n))$. If there exists a polynomial $p(n)$ such that $A$ is $O(p(n))$, we say that $A$ is \textit{polynomial time}.  For a more thorough definition of these concepts, we would direct the curious reader towards the standard~\cite[Chapter 2]{CLR}. Throughout the rest of this section, assume that $n=|V(H)|$, where $H$ is an $r$-uniform hypergraph. 

Gallo, Longo, Pallottino, and Nguyen~\cite{GLPN} provide an algorithm \textbf{Visit($H,v$)} which takes as input a hypergraph $H$ and a vertex $v \in V(H)$, and returns all vertices in $H$ that are in the same connected component as $v$. The \textbf{Visit} algorithm is $O(s)$, where $$s = \displaystyle\sum_{e \in E(H)} |e|.$$ If we assume that $H$ is an $r$-uniform hypergraph, then there are at most $\binom{n}{r}$ hyperedges, and each $e \in E(H)$ has $|e| = r$. Recall that if $r$ is a fixed constant, $\binom{n}{r}$ is $O(n^r)$ when viewed as a function of $n$.  Thus, in the case of an $r$-uniform hypergraph, we have $s = O(n^r),$ and \textbf{Visit} becomes a polynomial time algorithm.
Notice that since being in the same connected component is an equivalence relation, it is not hard to see that an $r$-uniform hypergraph $H$ is connected if and only if \textbf{Visit($H,v$)} returns the entire set $V(H)$ for any $v \in V(H)$. 

\begin{theorem}
For any value of $r$, there exists an algorithm that takes as input a connected $r$-uniform hypergraph $H$ and returns $M$, a minimally connected spanning subhypergraph of $H$. If $r$ is viewed as a fixed constant, this algorithm runs in polynomial time on the number of vertices in $H$. 
\end{theorem}

\begin{proof}
Consider the following algorithm:  Let $H^*$ be a copy of $H$. Begin by choosing an arbitrary vertex $v^* \in H$. For each edge $e$ in $E(H^*)$, run \textbf{Visit($H^* - \{e \}, v^*$)} to determine if $e$ can be removed without disconnecting $H^*$. If yes, remove $e$ from $E(H^*)$. 
Note that  $H^* - \{e\}$ can only be connected if $H^* - \{e\}$ will contain all of $V(H^*)$. So at each stage, $H^*$ remains a spanning subhypergraph of $H$. This algorithm will terminate with $H^*$ as a minimally connected subhypergraph of $H$, since if an edge of $E(H^*)$ could be removed without disconnecting $H^*$, that edge would have been removed at the stage it was considered. Since $H$ is an $r$-uniform hypergraph, there are at most $\binom{n}{r}$ hyperedges, which is $O(n^r)$, and for each hyperedge we call \textbf{Visit($H,v^*$)} which we know is an $O(n^r)$ algorithm. Hence the algorithm is $O(n^{2r})$, which is a polynomial in $r$. 
\end{proof}

It is worth emphasizing that these results provide a sharp contrast  between spanning minimally connected subhypergraphs and spanning trees. For instance, for a $4$-uniform hypergraph on $n$ vertices, we have just shown that there must be a spanning minimally connected subhypergraph, and in fact there is an $O(n^8)$ algorithm to find one. As discussed in the Introduction, $4$-uniform hypergraphs do not always have spanning trees, and no known polynomial time algorithm is known to find such a spanning tree if it does exist. We would hazard the following intuitive explanation for the potential gap in complexity between the problems. Our algorithm is able to consider each edge one at a time and check whether removal will result in a disconnected subhypergraph. The problem of finding a spanning tree adds the additional difficulty of avoiding a cycle, which is a more global property based on collections of edges rather than single edges. 






\section{Chromatic numbers of Minimally Connected Hypergraphs}\label{chromat}

Chromatic numbers give insight into the connectivity of a graph or hypergraph.  Let $\chi _w$ and $\chi _s$ denote the weak and strong chromatic numbers, respectively.  That is, $\chi _w(H)$ is the minimum number of colors needed to properly color the vertices of $H$ so that no hyperedge is monochromatic  and $\chi _s(H)$ is the minimum number of colors needed to color the vertices of $H$ so that every pair of adjacent vertices receive different colors. When these concepts are restricted to the case of $2$-graphs, they both agree with that of the chromatic number.

\begin{theorem}
If $H$ is a minimally connected $r$-uniform hypergraph, then $\chi _w(H)=2.$
\end{theorem}

\begin{proof}
Let $H$ be a minimally connected $r$-uniform hypergraph with $m$ edges. From Theorem \ref{equiv}, we know that a minimally connected hypergraph can be constructed hyperedge-by-hyperedge with a connected hypergraph each step of the way and such that each new hyperedge requires the addition of a new vertex.  Suppose that $H_i$ is the connected hypergraph formed after adding hyperedge $e_i$ ($1\le i \le m$).  We proceed by induction on $m$ to prove that $\chi _w (H)=2$.  $H_1$ consists of a single hyperedge, so its vertices can be trivially $2$-colored.  Now suppose that $H_i$ can be weakly $2$-colored.  When adding $e_{i+1}$ to construct $H_{i+1}$, there are three possibilities: the vertices in $E(H_i)\cap e_{i+1}$ are all the same color or receive both colors $1$ and $2$.  In the former case, give the new vertex added with $e_{i+1}$ the other color.  In the latter case, the new vertex can receive either color. In both cases, we find that $H_{i+1}$ can be weakly $2$-colored, and hence, $\chi_w(H)=2$.  
\end{proof}

While $r$-uniform trees with size $k\ge 1$ have strong chromatic number $\chi _s(T)=r$ (this is a simple inductive exercise to confirm), 
we are unable to provide such precise limitations on the strong chromatic number of hypergraphs that are only assumed to be minimally connected.  The following theorem demonstrates a method for finding minimally connected hypergraphs with arbitrarily large strong chromatic numbers.  Note that the following statement is not true for minimally connected graphs ($2$-uniform trees), so one must assume $r\ge 3$.

\begin{theorem}
\label{Thm4}
For all natural numbers $n\ge r\ge 3$, there exists a minimally connected $r$-uniform hypergraph with strong chromatic number equal to $n$.
\end{theorem}

\begin{proof}
We begin with a complete graph $K_n$ of order $n$, which has size $m=\frac{n(n-1)}{2}$ and chromatic number $\chi (K_n)=n$.   From this graph, we form an $r$-uniform hypergraph $H_n^{(r)}$ that is minimally connected by replacing each edge $ab$ in $K_n$ with an $r$-uniform hyperedge $e_i =abx^i_1x^i_2\cdots x^i_{r-2}$, where $x^i_j$ are new vertices (with $1\le i\le m$ and $1\le j\le r-2$) that all have degree one.  The hypergraph $H_n^{(r)}$ is minimally connected since the removal of $e_i$ leaves each $x^i_j$ disconnected from the rest of the hypergraph.  Also, $H_n^{(r)}$ requires at least $n$ colors in any proper coloring since adjacent vertices in $K_n$ are still adjacent in $H_n^{(r)}$.  Hence, $\chi _s(H_n^{(r)})\ge n$.  On the other hand, since $n\ge r$, there are enough colors from the original proper coloring of $K_n$ to color the vertices $x^i_1, x^i_2, \dots , x^i_{r-2}$ distinct from one another.  Thus, $\chi_s (H_n^{(r)})\le n$, completing the proof.\end{proof}

The construction in the above proof provides, for any natural number $n$, a means of producing a minimally connected $r$-uniform hypergraph with strong chromatic number at least $n$.  As an example, consider Figure \ref{strongchrom}.  In this figure, $H_4^{(3)}$ is constructed from $K_4$.  The hypergraph $H_4^{(3)}$ contains six hyperedges, the deletion of any hyperedge results in a disconnected hypergraph containing an isolated vertex, and $\chi _s (H_4^{(3)})=4$.

\begin{figure}[h!]
\centerline{
{\includegraphics[width=0.7\textwidth]{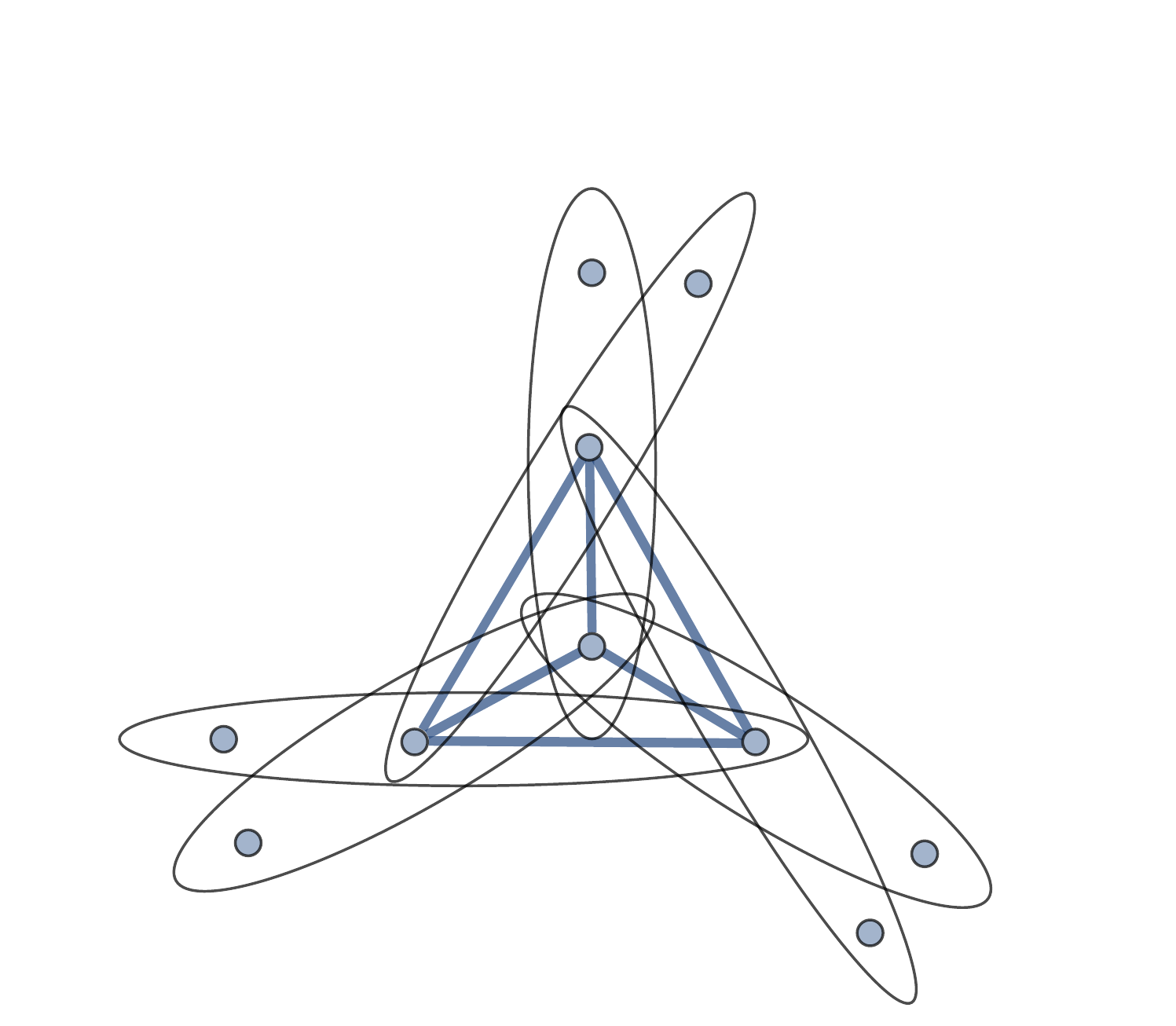}}}\caption{Using a $K_4$ to construct a minimally connected $3$-uniform hypergraph $H_4^{(3)}$ satisfying $\chi _s(H_4^{(3)})=4$.}\label{strongchrom}
\end{figure} 

\begin{theorem}
Let $H$ be a minimally connected $r$-uniform hypergraph with $r\ge 3$.  If $x$ and $y$ are two distinct vertices in $H$ that receive the same color in some strongly proper vertex coloring of $H$, then there exists some hyperedge whose removal puts $x$ and $y$ in different components.
\end{theorem}

\begin{proof}
Suppose that $x$ and $y$ receive the same color in some proper vertex coloring of $H$, but there does not exist any hyperedge whose removal puts $x$ and $y$ in different components.  Let $\mathcal{P}$ be a constructive process with hyperedges $e_1$, $e_2,$ \dots , $e_k$.  Then removing $e_k$ does not result in $x$ and $y$ being in different components.  Removing the hyperedges in reverse order, the removal of $e_i$ does not result in $x$ and $y$ being in different components for any $2\le i \le k$.  Finally, we are left with $x$ and $y$ both in $e_1$.  So, they cannot receive the same color, giving us a contradiction.  
\end{proof}

\begin{theorem}
For $r\ge 3$, let $H$ be a minimally connected $r$-uniform hypergraph of size $k>1$ with minimally connected constructive process $\mathcal{P}$ having tightness sequence $t_1$, $t_2$, \dots , $t_k$.  Then $$\chi_s (H)\le r+t_H-t_1-k+2.$$ \label{chrombound}
\end{theorem}

\begin{proof}
We proceed by induction on $k>1$.  When $k=2$, $t_H=t_1$, and it is easily seen that $\chi _s(H)=r$.
Now, suppose that the result is true for all minimally connected $r$-uniform hypergraphs of size $k-1$ and let $H$ be a minimally connected $r$-uniform hypergraph of size $k$.  Let $H'$ be the minimally connected hypergraph formed by removing $e_k$ (and all isolated vertices) from $H$.  Then by the inductive hypothesis, $$\chi_s(H')=r+\mathop{\sum}\limits_{i=1}^{k-2} t_i -t_1-(k-1)+2.$$  When adding in $e_k$, the only vertices that may need new colors are those that are the same color in $V(H')\cap e_k$.  So, at most, $t_{k-1}-1$ new colors are needed, from which it follows that $$\chi_s(H)\le r+\mathop{\sum}\limits_{i=1}^{k-1} t_i -t_1-k+2=r+t_H-t_1-k+2,$$ completing the proof of the theorem.
\end{proof}

\noindent Observe that we can optimize the bound in Theorem \ref{chrombound} by picking a constructive process in which $t_1$ is maximal. 

\section{Some Hyperedge Coloring Results}\label{C}

In this section we will examine colorings of the hyperedges of complete $r$-uniform hypergraphs.  A {\it $t$-coloring} of an $r$-uniform hypergraph $H$ is a function $$c: E(H)\longrightarrow \{1, 2, \dots, t\}$$ that assigns colors to the hyperedges of $H$.  We do not assume that such a coloring is proper, nor do we assume that $c$ is surjective.  A subgraph $H'$ of $H$ is called {\it rainbow} if all of the hyperedges in $H'$ receive different colors.  We call $H$ {\it rainbow connected with respect to $c$} if for every pair of distinct vertices $u,v\in V(H)$, there exists a rainbow Berge path connecting $u$ to $v$.  Observe that every connected hypergraph $H$ is rainbow connected with respect to some coloring as one could always choose the coloring in which every hyperedge of $H$ receives a different color.  The {\it rainbow connected number} $rc(H)$ of a connected $r$-uniform hypergraph $H$ is defined to be the minimal number of colors $t$ such that $H$ is rainbow connected with respect to some $t$-coloring.

Since every connected $r$-uniform hypergraph $H$ is spanned by a minimally connected subhypergraph $M$, the edges of $M$ can each receive a different color, providing a rainbow Berge path between every distinct pair of vertices.  Hence, from Theorem \ref{size}, it follows that $$rc(H)\le n-r+1.$$

Let $K_n^{(r)}$ denote the complete $r$-uniform hypergraph of order $n$.  A {\it Gallai $t$-coloring} of $K_n^{(r)}$ is a $t$-coloring of the hyperedges in $K_n^{(r)}$ such that no rainbow $K_{r+1}^{(r)}$-subhypergraph exists.  Thus, when $t\le r+1$ all $t$-colorings are Gallai $t$-colorings. 
The following theorem is a nice generalization of a result concerning Gallai colorings of graphs from Gy\'arf\'as and Simonyi \cite{GS}.

\begin{theorem}
Let $r\ge 3$.  Then every Gallai $(r+1)$-coloring of $K_n^{(r)}$ contains a color that spans a connected $r$-uniform hypergraph using all $n$ vertices.\label{Gall}
\end{theorem}

\begin{proof}
We proceed by induction on $n$.  When $n=r+1$, at least two hyperedges in $K_{r+1}^{(r)}$ are the same color by the definition of a Gallai coloring, and hence, span the complete hypergraph.  Now assume the theorem is true for $n\ge r+1$ and consider a Gallai coloring of $K_{n+1}^{(r)}$.  Let $\{x_1 , x_2 , \dots, x_{n+1} \}$ be the vertices of this $K_{n+1}^{(r)}$ and define $k_i$ to be the subhypergraph formed by removing $x_i$ from the $K_{n+1}^{(r)}$, for each $i\in \{ 1, 2, \dots , n+1 \}$.  Thus, we have a total of $n+1$ complete Gallai $(r+1)$-colored hypergraphs where $n+1\ge r+2$.  By the inductive hypothesis, each $k_i$ contains a spanning color, and by the pigeonhole principle, at least two of the $k_i$s are spanned by the same color.  This color necessarily spans all of $K_{n+1}^{(r)}$.  
\end{proof}
\noindent Combining Theorem \ref{span} with Theorem \ref{Gall}, it follows that every Gallai $(r+1)$-coloring of $K_n^{(r)}$ contains a monochromatic minimally connected spanning subhypergraph. 

We now turn our attention to the colorings of complete hypergraphs involving only two colors, but first, we must recall the definition of diameter.  Here, we only assume $r\ge 2$.  If $u$ and $v$ are any two distinct vertices in a connected $r$-uniform hypergraph $H$, then the {\it distance from $u$ to $v$,} denoted $d(u,v)$ is the minimum number of hyperedges contained in any Berge path connecting $u$ to $v$.  The {\it diameter} $diam(H)$ is then defined by $$diam(H):=\max\{ d(u,v)\ | \ u,v\in V(H)\}.$$  Results involving $2$ colorings in the setting of $2$-graphs are plentiful. In particular, it has been noted that  Erd\H{o}s and Rado claimed that for any graph $G$, either $G$ or $\overline{G}$ is connected \cite{BG}.  Rephrasing this result in terms of colorings, every $2$-coloring of $K_n$ contains a color that spans all $n$ vertices.  More precisely, it was shown by Bialostocki, Dierker, and Voxman \cite{BDV} that if a graph $G$ is not connected, then $\overline{G}$ is connected and $diam(\overline{G})\le 2$.  When $r\ge3$, we offer a generalization of this fact in Theorem \ref{Thm17}. But first we derive some helpful results. 

\begin{theorem}\label{Thm15}
If $r\ge 3$ and $H$ is a disconnected $r$-uniform hypergraph of order $n\ge r$, then $\overline{H}$ is connected and every subset $\{ x_1, x_2, \dots , x_{r-1}\}$ of distinct vertices in $V(H)$ is contained in some hyperedge in $\overline{H}$.
\end{theorem}

\begin{proof}
Suppose that $H$ is a disconnected $r$-uniform hypergraph of order $n\ge r$.  Let $S=\{ x_1, x_2, \dots , x_{r-1}\}$ be a subset distinct vertices in $V(H)$.  We consider two cases. \\
\underline{Case 1:}  Assume that two distinct vertices $x_i$ and $x_j$ are in different connected components in $H$.  Then if $y$ is any vertex not in $S$, then $x_1x_2\cdots x_{r-1}y\in E(\overline{H})$.\\
\underline{Case 2:} Assume that $C_1$ and $C_2$ are distinct connected components in $H$ with $x_1,x_2,\dots , x_{r-1} \in C_1$.   If $y$ is any vertex in $C_2$, then $x_1x_2\cdots x_{r-1}y\in E(\overline{H})$.\\
In both cases, we find that $\overline{H}$ is connected and every subset $\{x_1, x_2,\dots , x_{r-1}\}$ of vertices is contained in some hyperedge in $\overline{H}$.
\end{proof}

From Theorem \ref{Thm15}, observe that whenever $H$ is disconnected, it follows that $\overline{H}$ is connected and $diam(\overline{H})=1$.  Furthermore, using this theorem  with Theorem \ref{span}, we obtain the following corollary.

\begin{corollary}\label{cor1} Let $r\ge 3$.
In every $2$-coloring of the hyperedges in $K_n^{(r)}$, there exists a monochromatic spanning minimally connected hypergraph.
\end{corollary}

\begin{theorem}\label{Thm17}
If $r\geq 3$ and $H$ is a connected $r$-uniform hypergraph of order $n\ge r$ with $diam(H)\ge 2$, then $diam(\overline{H})= 1$.
\end{theorem}

\begin{proof}
Assume that $H$ has diameter at least $2$. Then by Corollary \ref{cor1} we can assume that $H$ is connected. We may choose $x,y\in V(H)$ to be nonadjacent vertices. Since $H$ is connected and has at least $r>2$ vertices then there are $r-2$ vertices $z_1,z_2,...,z_{r-2}$ such that for some $e\in E(H)$ $\{x,z_1,z_2,...,z_{r-2}\}\subset e.$ Then the edge $e_1=\{x,z_1,z_2,...,z_{r-2},y\} \not\in E(H)$. Thus $e_1 \in E(\overline{H}).$ Therefore $x$ and $y$ are adjacent in $\overline{H}$,  proving that $diameter(\overline{H})=1.$ 
\end{proof}

\section{Conclusions and future directions}

This paper has been an investigation into hypergraph generalizations of trees.  The specific generalization which we explored (minimally connected hypergraphs) stresses the role that trees play in results concerning connectivity. While minimally connected hypergraphs work well to expand upon many theorems for $2$-graphs, the fact that the size of a minimally connected spanning tree is not necessarily determined by the parent graph shows that many open problems still exist. We conclude by listing several open problems that we deem worthy of future investigation.

\begin{enumerate}
\item Is there an algorithm that would provide, for arbitrary weights, the optimum cost minimally connected spanning subhypergraph? The answer to this question could have serious ramifications, not just for bioinformatics (where punning algorithms have already been applied), but also in realms as distinct as physics and banking.   
Algorithms for exact and approximate minimum spanning trees in hypergraphs are already known (e.g., see \cite{GaS}, \cite{GZRSRB}, and \cite{GM}).  We could ask for conditions under which a weighted hypergraph would admit an optimal cost minimally connected subhypergraph, and the running time of an optimal algorithm for finding it. Figure \ref{minimalweight} shows that a greedy approach of taking the hyperedges in order of least weight can be suboptimal. Such an approach algorithm would yield a minimally connected subhypergraph with hyperedges $cde$, $bde$, and $ade$, for a total weight of $5$, whereas the hyperedges $cde$ and $abc$ would form a minimally connected subhypergraph with a lower total weight of $4$. 
\begin{figure}[h!]
\centerline{
{\includegraphics[width=0.6\textwidth]{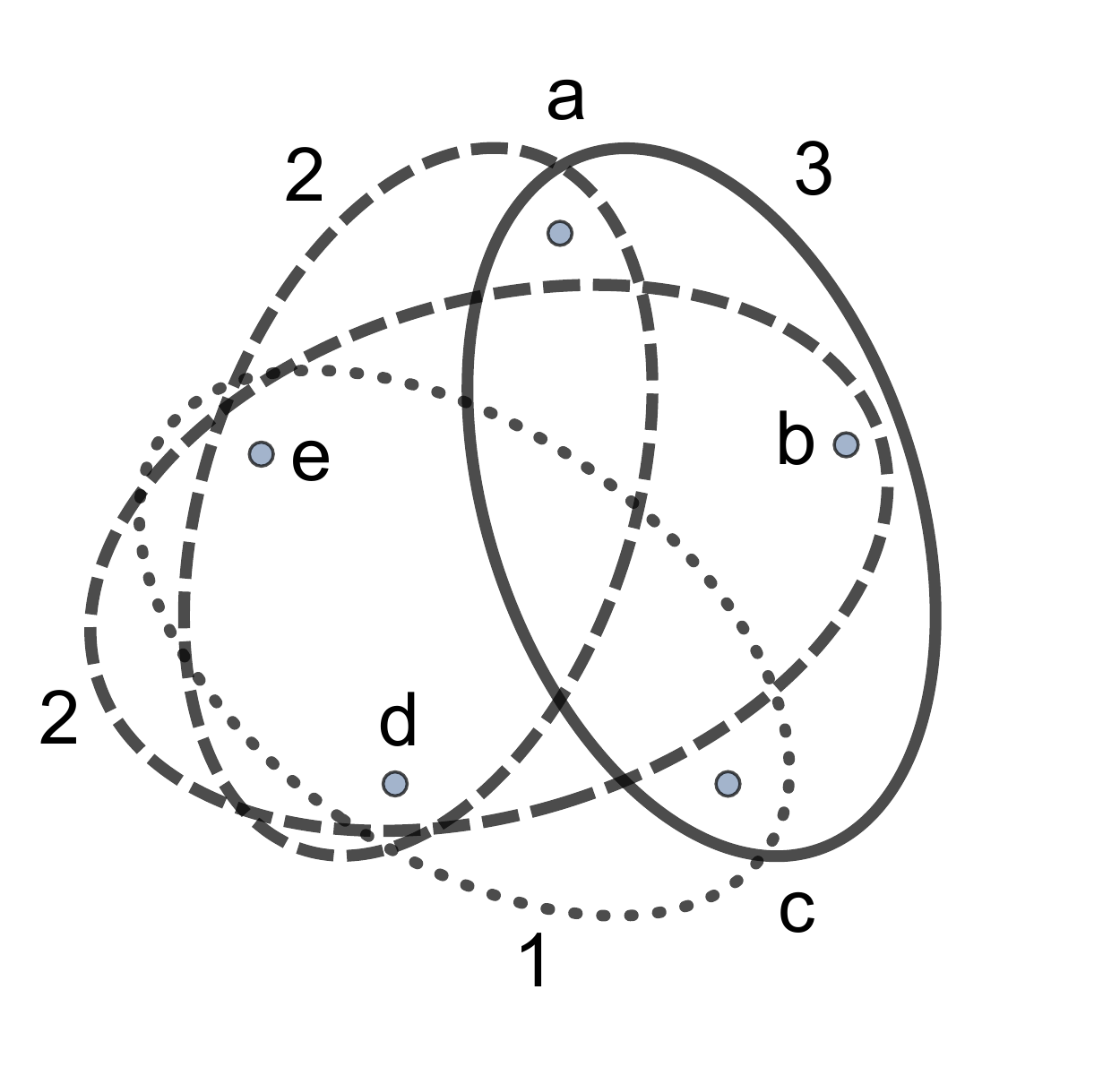}}}\caption{A weighted $3$-uniform hypergraph for which the greedy approach does not produce a minimally connected subhypergraph of minimum total edge weight.}\label{minimalweight}
\end{figure} 
\item Recall that the Ramsey number $R(H_1, H_2;r)$ of two $r$-uniform hypergraphs $H_1$ and $H_2$ is defined to be the least natural number $p$ such that every $2$-coloring of hyperedges of $K_p^{(r)}$, using say, red and blue, results in a red subhypergraph isomorphic to $H_1$ or a blue subhypergraph isomorphic to $H_2$.  A connected $r$-uniform hypergraph $H$ of order $m$ is called {\it $n$-good} if $$R(H, K_n^{(r)};r)=(m-1)\left( \ceil[\Big]{\frac{n}{r-1}}-1\right)+t(K_n^{(r)}),$$ where $\ceil{\cdot}$ is the ceiling function and $t(K_n^{(r)})$ is the minimum number of vertices in any color class of a weak vertex coloring of $H$.  The fact that this number is a lower bound for the given Ramsey number was proved in Theorem 3.1 of \cite{BP}.  So, showing that an $r$-uniform hypergraph is $n$-good follows from proving that this number is also an upper bound. In \cite{BP}, it was conjectured that all $r$-uniform trees are $n$-good and infinitely-many examples of $n$-good $3$-uniform trees were given.  It was also shown that the minimally connected $3$-uniform cycle $C_4^{(3)}$ of length $2$ and order $4$ is $4$-good, but is not $5$-good.  What are the conditions under which a minimally connected $r$-uniform hypergraph is $n$-good?  It is worth noting that if a connected $r$-uniform hypergraph is $n$-good, then so is every minimally connected spanning subhypergraph.
\item In Section \ref{chromat}, we considered the weak and strong chromatic numbers of minimally connected $r$-uniform hypergraphs, but other chromatic numbers can be considered when $r\ge 4$.  More generally, define the $k$-chromatic number $\chi _k(H)$ of an $r$-uniform hypergraph $H$ to be the minimum number of colors needed to color the vertices of $V(H)$ so that every hyperedge contains vertices using at least $k$ distinct colors.  It follows that $$\chi _w(H)=\chi _2(H) \quad \mbox{and} \quad \chi _s(H)=\chi_r(H).$$ Can one determine $\chi _k(H)$, when $2<k<r$ (assuming $r\ge 4$)?  We saw in Section \ref{chromat} that $\chi _w(H)=2$ when $H$ is minimally connected, but for every $n\ge r$, there exists a minimally connected $r$-uniform hypergraph $H$ such that $\chi_s(H)=n$.  For $r\ge 4$, what is the smallest value of $k$ for which $\chi _k$ is bounded for all minimally connected $r$-uniform hypergraphs?
\end{enumerate}

\bibliographystyle{amsplain}

\end{document}